\documentclass{amsart}

\newtheorem{theorem}{Theorem}[section]
\newtheorem{proposition}[theorem]{Proposition}
\newtheorem{lemma}[theorem]{Lemma}
\newtheorem{corollary}[theorem]{Corollary}
\newtheorem{fact}[theorem]{Fact}

\theoremstyle{definition}
\newtheorem{definition}[theorem]{Definition}

\theoremstyle{remark}
\newtheorem{remark}[theorem]{Remark}

\def\Ind{\setbox0=\hbox{$x$}\kern\wd0\hbox to 0pt{\hss$\mid$\hss} \lower.9\ht0\hbox to 0pt{\hss$\smile$\hss}\kern\wd0} 

\def\Notind{\setbox0=\hbox{$x$}\kern\wd0\hbox to 0pt{\mathchardef \nn=12854\hss$\nn$\kern1.4\wd0\hss}\hbox to 0pt{\hss$\mid$\hss}\lower.9\ht0 \hbox to 0pt{\hss$\smile$\hss}\kern\wd0} 

\def\ind{\mathop{\mathpalette\Ind{}}}

\newcommand{\compcent}[1]{\vcenter{\hbox{$#1\circ$}}} 

\newcommand{\comp}{\mathbin{\mathchoice 
{\compcent\scriptstyle}{\compcent\scriptstyle} 
{\compcent\scriptscriptstyle}{\compcent\scriptscriptstyle}}} 

\numberwithin{equation}{section}

\def \d {\delta}

\def \dd {\partial}
\def \D {\Delta}
\def \t {\theta}
\def \T {\Theta}
\def \l {\langle}
\def \r {\rangle}
\def \L {\Lambda}
\def \I {\mathcal I}
\def \P {\mathcal P}
\def \L {\Lambda}
\def \V {\mathcal V}
\def \NN {\mathbb N}
\def \QQ {\mathbb Q}
\def \U {\mathbb U}
\def \C {\mathcal C}

\def \al {\alpha}
\def \DD {\mathcal D}

\def \Ga {\Gamma}

\def \s {\sigma}

\def \U {\mathbb U}
\def \M {\mathcal M}
\def \w {\omega}
\def \HS {\underline{\mathcal{D}}}

\title[Partial differential fields with an automorphism]{On the model companion of partial differential fields
with an automorphism}
\author{Omar Leon Sanchez}
\address{Omar Leon Sanchez\\
McMaster University\\
Department of Mathematics and Statistics\\
1280 Main Street West\\
Hamilton, Ontario \  L8S 4L8\\
Canada}
\email{oleonsan@math.mcmaster.ca}
\date{\today}

\begin{document}

\begin{abstract}
We prove that the class of partial differential fields of characteristic zero with an automorphism has a model companion. We then establish the basic model theoretic properties of this theory and prove that it satisfies the canonical base property, and consequently the Zilber dichotomy, for finite dimensional types.
\end{abstract}

\maketitle

\begin{center}
\end{center}

\section*{Introduction}

The study of fields equipped with additive operators is a topic that has drawn the attention of the mathematical community since the early 1940's. For instance, the study of meromorphic solutions of differential equations led to the study of arbitrary fields equipped with several commuting derivations. This subject, now known as differential algebra, is an extremely rich area of study with active research in differential Galois theory and applications to computer algebra and numerical integration techniques. Similarly, the study of difference equations, in the realm of fields equipped with automorphisms, has had several applications to algebraic dynamics and diophantine geometry.

Even though the study of the above two areas is interesting in its own right, the interplay between them has also led to the development of the theory of differential-difference fields; that is, fields equipped with commuting derivations and automorphisms. In the 1970's, this theory was first studied from an algebraic perspective by Cohn \cite{Co}. Later on, a Galois theory for linear differential-difference equations was developed by Hardouin and Singer \cite{HaSi}, among others.

Whether it is in the differential, difference, or differential-difference setting, the main applications come from understanding the algebraic geometry that arises when considering the solution sets of differential, difference, and differential-difference equations, respectively. A natural approach to the subject is to work in a large universal domain where all the varieties, defined over various fields, cohabit; much in the spirit of Weil's algebraic geometry.

Model theory provides the right analogue notion of universal domains for the above settings. These domains are precisely the (large) existentially closed models of the theory under consideration; that is, the models that have a realization for every quantifier free formula for which a realization exists in some extension. In order to apply model-theoretic techniques (from geometric stability theory, for instance) it is customary to show that the class of existentially closed models is axiomatizable. Equivalently, to show that the theories of differential, difference, and differential-difference fields have a model companion.

The existence of the model companions of the theories of differential and difference fields has already proven its fruitfulness. For example, in the setting of (partial) differential fields of characteristic zero, the model-theoretic properties of the model companion, such as $\omega$-stability and the Zilber dichotomy, were at the heart of the development of a generalized differential Galois theory \cite{Pi9} and the proof of the Mordell-Lang conjecture for function fields \cite{Hru1}. Similarly, in the context of difference fields with one automorphism (it is known that the theory of difference fields with $n>1$ commuting automorphisms does not admit a model companion), again the model-theoretic properties of the model companion, such as supersimplicity and a version of the Zilber dichotomy, have been succesfully utilized in number theoretic applications such as the model-theoretic proof of the Manin-Mumford conjecture \cite{Hru2}.

In this paper we study the model theory of differential-difference fields with several commuting derivations and one automorphism. In particular, we show that this theory admits a model companion, and that it is supersimple and it satisfies the Zilber dichotomy. We expect that the results of this paper can be applied in the context of differential-difference Galois theory, and that (using the properties of canonical bases in supersimple theories) a version of Zilber's indecomposability theorem for definable groups holds.

It has been known for over fifteen years that the theory of ordinary differential fields of characteristic zero with an automorphism admits a model companion, and it has been studied intensively by Bustamante \cite{Bu}, \cite{Bu2}, \cite{Bu3}. However, the techniques used to prove the existence of the model companion do not extend to the theory of partial differential fields with an automorphism. In this paper we use a different approach to show that the latter theory has indeed a model companion, which we denote by $DCF_{0,m}A$. 

For many geometrically well behaved theories one can prove the existence of a model companion by adapting the axiomatization of $ACFA$ given by Chatzidakis and Hrushovski in terms of algebro-geometric objects \cite{CH}. These so called ``geometric axiomatizations'' have succesfully been applied to yield the existence of model companions in several interesting theories: ordinary differential fields \cite{PiPi}, partial differential fields (several commuting derivations) \cite{Le}, \cite{Le2}, fields with commuting Hasse-Schmidt derivations in positive characteristic \cite{Kow}, fields with free operators \cite{MS}, and in theories having a ``geometric notion of genericity'' \cite{Hi}.

In \cite{GR}, Guzy and Rivi\`ere point out that the existentially closed partial differential fields with an automorphism are characterized by a certain differential-algebro geometric condition (see Fact~\ref{GR} below), very much in the spirit of the geometric axioms for $ACFA$. However, it remains open as to whether their geometric condition can be expressed in a first order way. The problem lies in determining the definability of differential irreducibility and differential dominance in definable families of differential algebraic varieties (these problems are related to the generalized Ritt problem, see \cite{GKO} and \cite{HKM}). 

Motivated by the methods in \cite{Le2}, we bypass the above definability issues by applying the differential algebraic machinery of characteristic sets of prime differential ideals developed by Kolchin \cite{Ko} (and Rosenfeld \cite{Ro}). More precisely, we prove a characterization of the existentially closed models in terms of characteristic sets of prime differential ideals (see Theorem~\ref{main}). Then, using the fact that the condition of being a characteristic set of a prime differential ideal is definable  (see Fact~\ref{chr} below), we observe that our characterization is indeed first order.

Once we have the existence of the model companion $DCF_{0,m}A$, the results in \S3 of \cite{CP} imply, among other things, that each completion of $DCF_{0,m}A$ is supersimple and it satisfies the independence theorem over differential-difference fields that are differentially closed. In Section 3, we formally state these results and observe that, following the spirit of the arguments in \S1 of \cite{CH}, one can show that the independence theorem holds over algebraically closed differential-difference fields. We also present, in Proposition~\ref{fields}, some basic properties of the fixed field and the various fields of constants.

In Section 5, we prove a Zilber dichotomy theorem for $DCF_{0,m}A$ via a now standard approach using jet spaces. More precisely, following the ideas of Pillay and Ziegler from \cite{PZ}, we develop a notion of $(\D,\s)$-jet space for finite dimensional $(\D,\s)$-varieties (which in this setting turn out to agree with the generalized jet spaces of Moosa and Scanlon \cite{MS2}, the reader familiar with their work is referred to the Appendix for details). Then, in Theorem~\ref{cbp}, we use these $(\D,\s)$-jet spaces and the properties of finite dimensional $(\D,\s)$-modules to prove the canonical base property for finite dimensional types. Finally, a standard adaptation of the argument of Pillay \cite{Pi8} for compact complex manifolds, shows that the canonical base property implies (and is rather stronger than) the Zilber dichotomy.

\

\noindent {\it Acknowledgements.} I would like to thank Rahim Moosa and Bradd Hart for their helpful comments and suggestions on a previous version of this paper.

\

\section{Differential algebraic preliminaries}

\subsection{Differential algebra}

\

By a differential ring (field) we mean a ring (field) equipped with a finite set $\D=\{\d_1,\dots,\d_m\}$ of commuting derivations. Let us fix a differential field $(K,\D)$ of characteristic zero. The set of derivative operators is defined as the commutative monoid 
$$\T:= \{\d_m^{e_m}\cdots \d_1^{e_1}\colon \, e_i<\omega\}.$$
Let $x=(x_1,\dots,x_n)$ be a tuple of (differential) indeterminates, the set of algebraic indeterminates is $\T x=\{\t x_i: \t\in \T, i=1,\dots,n\}$. Then the ring of differential polynomials over $K$ is defined as $K\{x\}:=K[\T x]$. One can equip $K\{x\}$ with the structure of a differential ring by extending $\D$, using the Leibniz rule and defining $$\d_j(\d_m^{e_m}\cdots \d_1^{e_1}x_i)=\d_m^{e_m}\cdots \d_j^{e_j+1}\cdots \d_1^{e_1}x_i.$$ 

The \emph{canonical ranking} on the algebraic indeterminates $\T x$ is defined by 
\begin{displaymath}
\d_m^{e_m}\cdots \d_1^{e_1}x_i< \d_m^{r_m}\cdots \d_1^{r_1}x_j \iff \left(\sum e_k,i,e_m,\dots,e_1\right)<\left(\sum r_k,j,r_m,\dots,r_1\right)
\end{displaymath}
in the lexicographical order. Let $f\in K\{x\}\setminus K$. The \emph{leader} of $f$, $v_f$, is the highest ranking algebraic indeterminate that appears in $f$. The \emph{degree} of $f$, $d_f$, is the degree of $v_f$ in $f$. The \emph{rank} of $f$ is the pair $(v_f,d_f)$. If $g\in K\{x\}\setminus K$ we say that $g$ has \emph{lower rank} than $f$ if $rank(g)<rank(f)$ in the lexicograpical order.

Let $f\in K\{x\}\setminus K$. The \emph{separant} of $f$, $S_f$, is the formal partial derivative of $f$ with respect to $v_f$, that is 
$$S_f:=\frac{\partial f}{\partial v_f}\in K\{x\}.$$
The \emph{initial} of $f$, $I_f$, is the leading coefficient of $f$ when viewed as a polynomial in $v_f$, that is, if we write 
$$f=\sum_{i=0}^{d_f}g_i\cdot v_f^i$$
where $g_i\in K\{x\}$ and $v_{g_i}<v_f$, then $I_f=g_{d_f}$. Note that both $S_f$ and $I_f$ have lower rank than $f$. 

\begin{definition}
Let $f,g\in K\{x\}\setminus K$. We say $g$ is \emph{partially reduced} with respect to $f$ if no (proper) derivative of $v_f$ appears in $g$. If in addition the degree of $v_f$ in $g$ is less than $d_f$ we say that $g$ is \emph{reduced} with respect to $f$.
\end{definition}

A finite set $\L=\{f_1,\dots,f_s\}\subset K\{x\}\setminus K$ is said to be \emph{autoreduced} if for all $i\neq j$ we have that $f_i$ is reduced with respect to $f_j$. We will always write autoreduced sets in order of increasing rank, i.e., $rank(f_1)<\dots<rank(f_s)$. The canonical ranking on autoreduced sets is defined as follows: $\{g_1,\dots,g_r\}<\{f_1,\dots,f_s\}$ if and only if, either there is $i\leq r,s$ such that $rank(g_j)=rank(f_j)$ for $j<i$ and $rank(g_i)<rank(f_i)$, or $r>s$ and $rank(g_j)=rank(f_j)$ for $j\leq s$.

An ideal $I$ of $K\{x\}$ is said to be a \emph{differential ideal} if $\d f\in I$ for all $f\in I$ and $\d\in \D$. Given a set $A \subseteq K\{x\}$ the differential ideal generated by $A$ is denoted by $[A]$. It can be shown that every differential ideal $I$ of $K\{x\}$ contains a lowest ranking autoreduced set (\cite{Ko}, Chap. I, \S10), called a \emph{characteristic set} of $I$. 

\begin{remark}\label{cont}
If $\L$ is a characteristic set of $I$ and $g\in K\{x\}$ is reduced with respect to all $f\in\L$, then $g\notin I$. Indeed, if $g$ were in $I$ then $$\{f\in \L\colon rank(f)<rank(g)\}\cup\{g\}$$ would be an autoreduced subset of $I$ of lower rank than $\L$, contradicting the minimality of $\L$.
\end{remark}

Even though differential ideals are not in general generated by their characteristic sets, prime differential ideals are determined by these.

\begin{fact}[\cite{Ko}, Chap. IV, \S 9]
Let $\P$ be a prime differential ideal of $K\{x\}$ and $\L$ a characteristic set of $\P$. Then $$\P=[\L]:H_\L^{\infty}=\{g\in K\{x\}\colon \, H_\L^\ell \cdot g\in [\L] \text{ for some } \ell <\omega\},$$
where $\displaystyle H_\L=\prod_{f\in\L}S_f I_f$.
\end{fact}

\subsection{Differentially closed fields}\label{notat}

\

Let $\mathcal L_m$ be the language of differential rings with $m$ derivations, and $DF_{0,m}$ the $\mathcal L_m$-theory of differential fields of characteristic zero with $m$ commuting derivations. In \cite{Mc}, McGrail showed that this theory has a model companion, the theory of differentially closed fields $DCF_{0,m}$, and proved it is a complete $\omega$-stable theory that admits quantifier elimination and elimination of imaginaries. We also have the following facts:
\begin{itemize}
\item Suppose $(K,\D)\models DCF_{0,m}$ and $A\subseteq K$. Then $\operatorname{dcl}(A)$ equals the differential field generated by $A$ and $\operatorname{acl}(A)$ equals $\operatorname{dcl}(A)^{alg}$.
\item If $K$ is sufficiently saturated and $k\subseteq l$ are (small) differential subfields, then for any tuple $a$ from $K$ we have that $a\ind_k l$ if and only if the differential field generated by $a$ over $k$ is algebraically disjoint from $l$ over $k$.
\end{itemize}

Assume $(K,\D)\models DCF_{0,m}$. Given $A\subseteq K\{x\}$ and $V\subseteq K^n$, we let 
$$\V(A):=\{a\in K^n\colon f(a)=0 \text{ for all } f\in A\}$$ 
and 
$$\I(V/K)_\D:=\{f\in K\{x\}\colon f(a)=0 \text{ for all } a\in V\}.$$ 
Also, if $\L$ is a characteristic set of a prime differential ideal we let $$ \V^*(\L):=\V(\L)\setminus\V(H_\L).$$
Note that $\V(A)$ is a closed set in the Kolchin topology of $K^n$ (in other words, a $\D$-closed set). Recall that this topology is the differential analogue of the Zariski topology of $K^n$; that is, the $\D$-closed sets of $K^n$ are precisely the zero sets of finite systems of differential polynomial equations over $K$. By the differential basis theorem, the Kolchin topology is Noetherian and thus every $\D$-closed set has an irreducible decomposition.

The following basic properties of characteristic sets are essential for our characterization of the existentially closed partial differential fields with an automorphism (see Theorem \ref{main}).

\begin{proposition}\label{basic}
Assume $(K,\D)\models DCF_{0,m}$. Let $\P$ be a prime differential ideal of $K\{x\}$ and $\L$ a characteristic set of $\P$. Then $\V^*(\L)=\V(\P)\setminus\V(H_\L)$, $\V^*(\L)\neq \emptyset$ and $\V(\P)$ is an irreducible component of $\V(\L)$.
\end{proposition}
\begin{proof}
Let $a\in\V^*(\L)$, we need to show that $f(a)=0$ for all $f\in \P$. Clearly if $f\in [\L]$ then $f(a)=0$. Let $f\in \P$, since $\P=[\L]:H_\L^\infty$, we can find $\ell$ such that $H_\L^\ell\, f\in [\L]$. Hence, $H_\L(a)f(a)=0$, but $H_\L(a)\neq 0$, and so $f(a)=0$. The other containment is clear. If $\V(\P)\subseteq \V(H_\L)$ then, by the differential Nullstellensatz, $H_\L$ would be in $\P$ contradicting Remark~\ref{cont}. Thus $\V(\P)\setminus\V(H_\L)\neq \emptyset$. Now let $V$ be an irreducible component of $\V(\L)$ containing $\V(\P)$.  Since $\V^*(\L)=\V(\P)\setminus \V(H_\L)$, we have that $V\setminus\V(H_\L)=\V(\P)\setminus \V(H_\L)\subseteq \V(\P)$. Hence $\V(\P)$ contains a nonempty $\D$-open set of $V$, and so, by irreducibility of $V$, $\V(\P)=V$.
\end{proof}

In order to prove the existence of the model companion for partial differential fields with an automorphism we will make use of the following result of Tressl.

\begin{fact}[\cite{Tr}, \S 4]\label{chr}
The condition that ``$\L=\{f_1,\dots,f_s\}$ is a characteristic set of a prime differential ideal" is a definable property (in the language $\mathcal L_m$) of the coefficients of $f_1,\dots,f_s$. More precisely, for any $\{f_1(u,x),\ldots,f_s(u,x)\}\subset \mathbb{Q}[u]\{x\}$, where $u=(u_1,\ldots,u_r)$ are (algebraic) indeterminates, there is an $\mathcal{L}_m$-formula $\phi$ such that for all $(K,\D)\models DF_{0,m}$ and $a\in K^r$ we have that $(K,\D)\models \phi(a)$ if and only if $\{f_1(a,x),\ldots,f_s(a,x)\}$ is a characteristic set of a prime differential ideal of $K\{x\}$.
\end{fact}

Let us comment that Tressl's proof is essentially an application of Rosenfeld's criterion (\cite{Ko}, Chap. IV, \S 9). This criterion reduces the problem of checking if a family $\{f_1,\dots,f_s\}\subset K\{x\}$ is a characteristic set of a prime differential ideal, to the classical problem of checking primality of ideals in polynomial rings in finitely many variables where uniform bounds are well known \cite{VS}.

\

\section{The model companion}

In this section we show that the theory of partial differential fields of characteristic zero with an automorphism has a model companion. The scheme of axioms we use are in terms of characteristic sets of prime differential ideals. We carry on the differential algebraic notation and terminology from the previous section.

We work in the language of differential rings with $m$ derivations and an endomorphism $\mathcal L_{m,\s}$. We denote by $DF_{0,m,\s}$ the theory of differential fields of characteristic zero with $m$ commuting derivations and an automorphism commuting with the derivations. By a differential-difference field we mean a field $K$ equipped with a set of derivations $\D=\{\d_1,\dots,\d_m\}$ and an automorphism $\s$ such that $(K,\D,\s)\models DF_{0,m,\s}$. In particular, our difference fields are all inversive.

We will use the following geometric characterization of the existentially closed models of $DF_{0,m,\s}$. 

\begin{fact}\label{GR}
Let $(K,\D,\s)$ be a differential-difference field. Then $(K,\D,\s)$ is existentially closed if and only if the following conditions hold
\begin{enumerate}
\item [(i)] $(K,\D)\models DCF_{0,m}$
\item [(ii)] Suppose $V$ and $W$ are irreducible $\D$-closed sets such that $W\subseteq V\times V^\s$ and $W$ projects $\D$-dominantly onto both $V$ and $V^\s$. If $O_V$ and $O_W$ are nonempty $\D$-open sets of $V$ and $W$, respectively, then there is $a\in O_V$ such that $(a,\s a)\in O_W$.
\end{enumerate} 
\end{fact}

If $V$ is a $\D$-closed set then $V^\s$ is the $\D$-closed set with defining differential ideal $\{f^\s:f\in \I(V/K)_\D\}$, where $f^\s$ is the differential polynomial obtained by applying $\s$ to the coefficients of $f$.

\begin{proof}
A proof of this in the case of $O_V=V$ and $O_W=W$ appears in \cite{GR}. The argument given there can easily be adapted (in the usual way), but we include a proof for the sake of completeness.

Suppose $(K,\D,\s)$ is existentially closed and $V$, $W$, $O_V$ and $O_W$ are as in condition (ii). Working in a sufficiently saturated elementary extension $(\U,\D)$ of $(K,\D)$, we can find a $\D$-generic point $(c,d)$ of $W$ over $K$ (i.e., a tuple from $\U$ such that $\I((c,d)/K)_\D=\I(W/K)_\D$). Clearly $(c,d)\in O_W$ and, since $W$ projects $\D$-dominantly onto $V$ and $V^s$, $c$ and $d$ are generic points of $V$ and $V^\s$, respectively, over $K$. Thus, $c\in O_V$. Because $DCF_{0,m}$ has quantifier elimination we have that $tp_{\D}(d/K)=\sigma(tp_\D(c/K))$, where $tp_\D(a/K)$ denotes the type of $a$ over $K$ in the language of differential rings $\mathcal{L}_m$. Hence, there is an automorphism $\s'$ of $(\U,\D)$ extending $\s$ such that $\s'(c)=d$. Since $(K,\D,\s)$ is existentially closed we can find a point in $K$ with the desired properties.

Now suppose conditions (i) and (ii) are satisfied. Let $\phi(x)$ be a conjunction of atomic $\mathcal L_{m,\s}$-formulas over $K$. Suppose $\phi$ has a realisation $a$ in some differential-difference field $(F,\D,\s)$ extending $(K,\D,\s)$. Let 
$$\phi(x)=\psi(x,\s x,\dots,\s^r x)$$
 where $\psi$ is a conjunction of atomic $\mathcal L_{m}$-formulas over $K$ and $r>0$. Let $b$ be the tuple $(a,\s a,\dots,\s^{r-1}a)$, $X$ be the $\D$-locus of $b$ in $F^{nr}$ over $K$, and $Y$ be the $\D$-locus of $(b,\s b)$ in $F^{2nr}$ over $K$. Let 
$$\chi(x_0,\dots,x_{r-1},y_0,\dots, y_{r-1}):= \psi(x_0,\dots,x_{r-1},y_{r-1})\wedge \left(\bigwedge_{i=1}^{r-1} x_i=y_{i-1}\right)$$
 then $\chi$ is realised by $(b,\s b)$. Since $(b,\s b)$ is a $\D$-generic point of $Y$, $b$ is a $\D$-generic point of $X$ and $\s b$ is a $\D$-generic point of $X^\s$, over $K$, we have that $Y$ projects $\D$-dominantly onto both $X$ and $X^\s$ over $K$. Thus, since $(K,\D)\models DCF_{0,m}$, $Y(K)$ projects $\D$-dominantly onto both $X(K)$ and $X^\s(K)$. Applying (2) with $V=O_V=X(K)$ and $W=O_W=Y(K)$, there is $c$ in $V$ such that $(c,\s c)\in W$. Let $c=(c_0,\dots,c_{r-1})$, then $(c_0,\dots,c_{r-1},\s c_0,\dots, \s c_{r-1})$ realises $\chi$. Thus, $(c_0,\s c_0,\dots,\s^r c_0)$ realises $\psi$. Hence, $c_0$ is a tuple from $K$ realising $\phi$. This proves $(K,\D,\s)$ is existentially closed.
\end{proof}

\begin{remark}\label{rem2}
It is not known if condition (ii) is expressible in a first order way. One of the problems lies on the fact that it is not known if differential irreducibility is a definable condition. That is, given a differential algebraic family of $\D$-closed sets $V_u$ in a differentially closed field $(K,\D)$, it is not known if the set 
$$\{a\in K^r\colon V_a \text{ is irreducible}\}$$
is definable (see \cite{GKO} and \cite{HKM} for partial results around this problem). The other problem is that it is not known if differential dominance onto irreducible $\D$-closed sets is a definable condition. That is, given a differential algebraic family $f_{u,v}: W_v\to V_u$ of differential polynomial maps between irreducible $\D$-closed sets in a differentially closed field $(K,\D)$, it is not known if the set 
$$\{(a,b)\in K^r\times K^s\colon f_{a,b} \text{ is $\D$-dominant}\}$$
is definable.
\end{remark}

We bypass these definability issues by showing a new characterization of the existentially closed models in terms of characteristic sets of prime differential ideals. Recall, from Section \ref{notat}, that if $\L$ is a characteristic set of a prime differential ideal  we let $\V^*(\L)=\V(\L)\setminus\V(H_\L)$.

\begin{theorem}\label{main}
Let $(K,\D,\s)$ be a differential-difference field. Then $(K,\D,\s)$ is existentially closed if and only if the following conditions hold
\begin{enumerate}
\item $(K,\D)\models DCF_{0,m}$
\item Suppose $\L$ and $\Ga$ are characteristic sets of prime differential ideals of $K\{x\}$ and $K\{x,y\}$, respectively, such that $$\V^*(\Ga)\subseteq \V(\L)\times \V(\L^\s).$$ Suppose $O$ and $Q$ are nonempty $\D$-open subsets of $\V^*(\L)$ and $\V^*(\L^\s)$, respectively, such that $O\subseteq \pi_x(\V^*(\Ga))$ and $Q\subseteq \pi_y(\V^*(\Ga))$. Then there is $a\in \V^*(\L)$ such that $(a,\s a)\in \V^*(\Ga)$.
\end{enumerate}
\end{theorem}

Here, if $\L=\{f_1,\dots,f_s\}$, then $\L^\s=\{f_1^\s,\dots,f_s^\s\}$, and $\pi_x$ and $\pi_y$ denote the canonical projections from $\V(\L)\times\V(\L^\s)$ onto $\V(\L)$ and $\V(\L^\s)$, respectively.

\begin{proof}
Suppose $(K,\D,\s)$ is existentially closed, and let $\L$, $\Ga$, $O$ and $Q$ be as in condition (2). Then $\L$ is a characteristic set of the prime differential ideal $\P=[\L]:H_\L^\infty$ and, by Proposition~\ref{basic}, $\V^*(\L)=\V(\P)\setminus \V(H_\L)$. So $O_V:=\V^*(\L)$ is a nonempty $\D$-open subset of the irreducible $\D$-closed set $V:=\V(\P)$. Similarly, $O_W:=\V^*(\Ga)$ is a nonempty $\D$-open subset of the irreducible $\D$-closed set $W:=\V\left([\Ga]:H_\Ga^\infty\right)$.

Next we show that $W\subseteq V\times V^\s$. By Proposition~\ref{basic} $$W\setminus \V(H_\Ga)=\V^*(\Ga)\subseteq \V(\L)\times \V(\L^\s),$$ so that by taking $\D$-closures $W\subseteq \V(\L)\times\V(\L^\s)$. Since $W$ is irreducible it must be contained in some $X\times Y$ where $X$ and $Y$ are irreducible components of $\V(\L)$ and $\V(\L^\s)$, respectively. Hence $O\subseteq\pi_x(W)\subseteq X$. On the other hand, $O\subseteq O_V\subset V$, and so $V=X$. Similarly, since $V^\s=\V([\L^\s]:H_{\L^\s}^{\infty})$ and working with $Q$ rather than $O$, we get $Y=V^\s$. Therefore, $W\subseteq V\times V^\s$. Since $O\subseteq \pi_x(W)$ and $Q\subseteq \pi_y(W)$, $W$ projects $\D$-dominantly onto both $V$ and $V^\s$. Applying (ii) of Fact~\ref{GR} to $V$, $W$, $O_V$, $O_W$, we get $a\in O_V$ such that $(a,\s a)\in O_W$, as desired.

For the converse we assume that (2) holds and aim to prove condition (ii) of Fact~\ref{GR}. In fact, it suffices to prove this statement in the case when $O_V=V$ and $O_W=W$ (see the comment at the beginning of the proof of Fact \ref{GR}). We thus have irreducible $\D$-closed sets $V$ and $W\subseteq V\times V^\s$ such that $W$ projects $\D$-dominantly onto both $V$ and $V^\s$, and we show that there is $a\in V$ such that $(a,\s a)\in W$. Let $\L$ and $\Ga$ be characteristic sets of $\I(V/K)_\D$ and $\I(W/K)_\D$, respectively. Then, by Proposition~\ref{basic} $$\V^*(\Ga)=W\setminus\V(H_\Ga)\subseteq V\times V^\s\subseteq \V(\L)\times\V(\L^\s).$$ Since $W$ projects $\D$-dominantly onto $V$ and $V^\s$, $\V^*(\Ga)$ projects $\D$-dominantly onto both $V$ and $V^\s$. Thus, by quantifier elimination for $DCF_{0,m}$, there are nonempty $\D$-open sets $O$ and $Q$ of $\V^*(\L)$ and $\V^*(\L^\s)$,  respectively, such that $O\subseteq \pi_x(\V^*(\Ga))$ and $Q\subseteq \pi_y(\V^*(\Ga))$. We are in the situation of 
condition~(2), and there is $a\in \V^*(\L)\subseteq V$ such that $(a,\s a)\in \V^*(\Ga)\subseteq W$.
\end{proof}

What the above proof shows is that, in a differentially closed field, each instance of condition (ii) of Fact~\ref{GR} is equivalent to an instance of condition (2) of Theorem~\ref{main}. This is accomplished by the passing from prime differential ideals to their characteristics sets, and from $\D$-dominant projections to containment of a nonempty $\D$-open set. However, these two characterizations have a very different flavour in terms of first order axiomatizability. More precisely, our new characterization has the advantage that condition (2) is expressible in a first order way. Indeed, suppose $(K,\D)\models DCF_{0,m}$ and let $\{f_1(u,x),\ldots,f_s(u,x)\}\subset \mathbb{Q}[u]\{x\}$, where $u=(u_1,\ldots,u_r)$ are (algebraic) indeterminates; while it is not known if the set 
$$\{a\in K^r\colon \V(f_1(a,x),\ldots,f_s(a,x)) \text{ is irreducible}\}$$
is definable, Fact \ref{chr} tells us precisely that
$$\{a\in K^r\colon \{f_1(a,x)\ldots,f_s(a,x)\} \text{ is a characteristic set of a prime $\D$-ideal}\}$$
is in fact definable. It follows easily now that condition (2) is first order. (Observe that by introducing the $\D$-open sets $O$ and $Q$ of condition (2) in our scheme of axioms we have avoided the issue of $\D$-dominance mentioned in Remark \ref{rem2}.)

By the above discussion, Theorem~\ref{main} and Fact~\ref{chr} imply the following:

\begin{corollary}\label{mc}
The theory $DF_{0,m,\s}$ has a model companion. 
\end{corollary}

Henceforth we denote this model companion by $DCF_{0,m}A$.

\

\section{Basic model theory of $DCF_{0,m}A$}

In this section we present some of the model theoretic properties of the theory $DCF_{0,m}A$. Many of these results are consequences of the work of Chatzidakis and Pillay in \cite{CP} or more or less immediate adaptations of the arguments from \cite{Bu} or \S5 of \cite{MS}.

Let $(K,\D,\s)$ be a differential-difference field and $A\subseteq K$. The differential-difference field generated by $A$, denoted by $\l A\r$, is the smallest differential-difference subfield of $K$ containing $A$. Note that $\l A\r$ is simply the subfield of $K$ generated by 
$$\{\d_m^{e_m}\cdots \d_1^{e_1}\s^k a \colon \, a\in A, e_i< \omega, k\in \mathbb{Z}\}.$$ 
If $k$ is a differential-difference subfield of $K$, we write $k\l B\r$ instead of $\l k\cup B\r$.

The following are consequences of \S3 of \cite{CP} and the fact that the model companion $DCF_{0,m}A$ exists:

\begin{proposition}\label{propert}
Let $(K,\D,\s)$ and $(L,\D',\s')$ be models of $DCF_{0,m}A$.
\begin{itemize}
\item [(i)] If $A\subseteq K$, then $\operatorname{acl}(A)=\l A\r ^{alg}$.
\item [(ii)] Suppose $K$ and $L$ have a common algebraically closed differential-difference subfield $F$, then $(K,\D,\s)\equiv_F(L,\D',\s')$. In particular, the completions of $DCF_{0,m}A$ are determined by the difference field structure on $\mathbb{Q}^{alg}$.
\item [(iii)] Suppose $k$ is a differential-difference subfield of $K$ and $a$, $b$ are tuples from $K$. Then $tp(a/k) = tp(b/k)$ if and only if there is an $k$-isomorphism from $(k\l a\r ^{alg},\D,\s)$ to $(k\l b\r^{alg},\D,\s)$ sending $a$ to $b$.
\item [(iv)] Every $\mathcal{L}_{m,\s}$-formula $\phi(x_1,\dots,x_n)$ is equivalent, modulo $DCF_{0,m}A$, to a
disjunction of formulas of the form $\exists y \psi(x_1, \dots, x_n, y)$, where $\psi$ is quantifier free, and such that for every tuple $(a_1,\dots,a_n,b)$ if $\psi(a_1,\dots,a_n,b)$ holds then $b\in \l a_1,\dots,a_n\r ^{alg}$.
\item [(v)] Every completion of $DCF_{0,m}A$ is supersimple. Moreover, if $(K, \D, \s)$ is sufficiently saturated and $A$, $B$, $C$ are small subsets of $K$, then $A \ind_C B$ if and only if $\l A \cup C\r$ is algebraically disjoint from $\l B \cup C\r$ over $\l C\r$.
\end{itemize}
\end{proposition}

By (v) of Proposition~\ref{propert}, $DCF_{0,m}A$ satisfies the independence theorem over models. Furthermore, by adapting the proof of Theorem 3.31 of \cite{Bu} or of Theorem 5.9 of \cite{MS}, one can show that $DCF_{0,m}A$ satisfies the independence theorem over algebraically closed differential-difference fields. That is, 

\begin{proposition}\label{indal}
Let $(K,\D,\s)$ be a sufficiently saturated model of $DCF_{0,m}A$. Suppose $F$ is a (small) algebraically closed differential-difference subfield of $K$, and 
\begin{enumerate}
\item [(i)] $A$ and $B$ are (small) supersets of $F$ with $A\ind_F B$, and
\item [(ii)] $a$ and $b$ are tuples such that $tp(a/F)=tp(b/F)$ and $a\ind_F A$ and $b\ind_F B$.
\end{enumerate}
Then there is $d\ind_F A\cup B$ with $tp(d/A)=tp(a/A)$ and $tp(d/B)=tp(b/B)$.
\end{proposition}

Thus, complete types over algebraically closed differential-difference fields are amalgamation bases, and hence each such type has a canonical base (see \cite{HKP} or \cite{Ca}). In other words, if $p$ is a complete type over an algebraically closed differential-difference field $F$, then there exists a set $Cb(p)$ such that $B:=\operatorname{dcl}(Cb(p))$ is the smallest definably closed subset of $F$ such that $p$ does not fork over $B$, and the restriction $p|_B$ is again an amalgamation bases. (The fact that we can take $Cb(p)$ in the real sort follows from Proposition~\ref{elimim} below, and the fact that supersimple theories eliminate hyperimaginaries \cite[Theorem 20.4]{Ca}.)

The proof of Proposition~3.36 of \cite{Bu} extends to the partial case to yield (one can also adapt the proof of Theorem 5.13 of \cite{MS}):

\begin{proposition}\label{elimim}
Every completion of $DCF_{0,m}A$ eliminates imaginaries.
\end{proposition}

Let $(K, \D, \s)$ be a differential-difference field. We denote by $K^\s$ the \emph{fixed field} of $K$, that is $K^\s =\{a\in K:\, \s a=a\}$, and by $K^\D$ we denote the \emph{field of (total) constants} of $K$, that is $K^\D=\{a\in K: \, \d a=0 \text{ for all } \d\in \D\}$. We let $\C_K$ be the field $K^\s\cap K^\D$.

More generally, if $\D'$ is a set of linearly independent elements of the $\C_K$-vector space $\operatorname{span}_{\C_K}\D$, we let $K^{\D'}$ be the field of \emph{$\D'$-constants} of $K$, that is $K^{\D'} =\{a\in K:\, \d a=0\text{ for all } \d\in \D'\}$. In particular, $K^\emptyset = K$ and if $\D'$ is a basis of $\operatorname{span}_{\C_K}\D$ then $K^{\D'}=K^\D$. Note that both $K^\s$ and $K^{\D'}$ are differential-difference subfields of $(K,\D,\s)$. Also, $(K,\D',\s)$ is itself a differential-difference field.

In the following proposition $DCF_{0,0}A$ stands for $ACFA_0$.

\begin{proposition}\label{fields}
Let $(K, \D, \s)\models DCF_{0,m}A$, $\D_1$ and $\D_2$ disjoint sets such that $\D':=\D_1\cup\D_2$ forms a basis of $\operatorname{span}_{\C_K}\D$, and $r = |\D_1|$. Then 
\begin{enumerate}
\item [(i)] $(K,\D',\s)\models DCF_{0,m}A$
\item [(ii)] $(K,\D_1,\s)\models DCF_{0,r}A$
\item [(iii)] $(K^{\D_2},\D_1,\s)\models DCF_{0,r}A$
\item [(iv)] $K^{\D_2} \cap K^\s$ is a pseudofinite field. 
\item [(v)] For all $k\geq 1$, $(K,\D,\s^k)\models DCF_{0,m}A$ 
\item [(vi)] $((K^\s)^{alg}, \D)\models DCF_{0,m}$
\end{enumerate}
\end{proposition}

\begin{proof} 
\

\noindent (i) It is easy to see that a set $V\subseteq K^n$ is $\D$-closed if and only if it is $\D'$-closed. Hence, irreducibility in the $\D$-topology is equivalent to irreducibility in the $\D'$-topology. Similarly a projection (onto any set of coordinates) is $\D$-dominant if and only if it is $\D'$-dominant. Therefore, each instance of the axioms of $DCF_{0,m}A$ (or rather of the characterization given by Fact~\ref{GR}) that needs to be checked for $\D'$ is true, as it is true for $\D$. 

\noindent (ii) By (1) we may assume that $\D_1\subseteq \D$. By the relative version of Kolchin's Irreducibility Theorem (\cite{Ko2}, Chapter 0, \S 6), every $\D_1$-closed set irreducible in the $\D_1$-topology is also irreducible in the $\D$-topology. Hence, every instance of the characterization of $DCF_{0,r}A$ given by Fact \ref{GR} that needs to checked for $\D_1$ is an instance of the characterization of $DCF_{0,m}A$ for $\D$. But all these instances are true since $(K,\D,\s)\models DCF_{0,m}A$.

\noindent (iii) By (1) we may assume that $\D_1\cup\D_2= \D$. We show that $(K^{\D_2},\D_1,\s)$ is existentially closed. Let $\phi(x)$ be a quantifier free $\mathcal{L}_{r,\s}$-formula over $K^{\D_2}$ with a realisation $a$ in some differential-difference field $(F, \Omega_1, \gamma)$ extending $(K^{\D_2} , \D_1, \s)$. Let $\Omega := \Omega_1\cup \{\rho_1,\dots, \rho_{m-r}\}$ where each $\rho_i$ is the trivial derivation on $F$. Hence, $(F, \Omega, \gamma)$ is a differential-difference field extending $(K^{\D_2} , \D, \s)$. Let $(L, \Omega, \gamma)$ be a model of $DCF_{0,m}A$ extending $(F, \Omega, \gamma)$. Since $K^{\D_2}$ is a common algebraically closed differential-difference subfield of $K$ and $L$, and $L \models \phi(a)$ and $\rho_i a = 0$ for $i=1,\dots,m-r$, by (ii) of Proposition~\ref{propert} $K$ has a realisation $b$ of $\phi$ such that $b\in K^{\D_2}$. Thus, since $\phi$ is quantifier free, $K^{\D_2}\models \phi(b)$.

\noindent (iv) By (3), we have that $(K^{\D_2}, \D_1, \s)\models DCF_{0,r}A$, and so, by (2), $(K^{\D_2},\s)\models ACFA_0$. Hence, the fixed field of $(K^{\D_2},\s)$, which is $K^{\D_2}\cap K^{\s}$, is pseudofinite.

\noindent (v) This can be shown as in Corollary 3.38 of \cite{Bu}. We give a slightly more direct argument using the characterization given by Fact~\ref{GR}. Let $V\subseteq K^n$ and $W\subseteq K^{2n}$ be irreducible $\D$-closed sets such that $W\subseteq V\times V^{\s^k}$ and $W$ projects $\D$-dominantly onto both $V$ and $V^{\s^k}$. We aim to show that there is $a\in V$ such that $(a,\s^k a)\in W$. Let 
$$\tilde{V}=V\times V^\s\times\cdots\times V^{\s^{k-1}}$$ 
and 
\begin{align*}
\tilde{W}=\{(x_0,\dots,x_{k-1},y_0,\dots,y_{k-1})\in K^{2kn}
&\colon (x_0,y_{k-1})\in W,\, (x_0,\dots,x_{k-1})\in \tilde{V}, \\
& \;\;\; \text{and } x_{i+1}=y_i  \text{ for }\,  i=0,\dots,k-2\}. 
\end{align*}
Then $\tilde{W}$ is an irreducible $\D$-closed subset of $\tilde{V}\times \tilde{V}^\s$ that projects $\D$-dominantly onto both $\tilde{V}$ and $\tilde{V}^\s$. Then, since $(K,\D,\s)\models DCF_{0,m}A$, we get a point in $\tilde{W}$ of the form
$$(a_0,\dots,a_{k-1},\s a_0,\dots,\s a_{k-1}).$$ By the equations of $\tilde{W}$, we have that $(a_0,\s a_{k-1})\in W$ and $a_{k-1}=\s^{k-1}a_0$. Hence, $(a_0,\s^k a_0)\in W$, as desired.

\noindent (vi) This appears in Theorem 3.2 of \cite{GR}. Let us give a brief sketch of the proof. Let $V$ be an irreducible $\D$-closed set defined over $(K^\s)^{alg}$. We need to show that $V$ has a $(K^\s)^{alg}$-point. Since the unique extension of $K^\s$ of degree $n$ is $K^{\s^n}$, we can find $k\geq 1$ such that $V^{\s^k}= V$. Let $W$ be the diagonal in $V \times V = V \times V^{\s^k}$, then $W$ projects $\D$-dominantly onto $V$ and $V^{\s^k}$. Hence, by (v), there is a point $a$ in $V$ such that $(a,\s^k a)$ is in $W$, and so $\s^k a=a$. Hence, $a$ is in $K^{\s^k} \subset (K^\s)^{alg}$.
\end{proof}

\

\section{The canonical base property and the Zilber dichotomy}\label{Zil}

In this section we prove the canonical base property, and consequently the Zilber dichotomy, for finite dimensional types in $DCF_{0,m}A$. Our proof follows the arguments given by Pillay and Ziegler in \cite{PZ}, where they prove the dichotomy for $DCF_0$ and $ACFA_0$ using the theory of jet spaces for algebraic varieties (this is also the strategy of Bustamante in \cite{Bu3} to prove the dichotomy for finite dimensional types in $DCF_{0,1}A$). 

For the rest of this section fix a sufficiently saturated $(\U,\D,\s)\models DCF_{0,m}A$, and an algebraically closed differential-difference subfield $K$ of $\U$. We first recall the theory of jet spaces from algebraic geometry. We refer the reader to \S 5 of \cite{MS1} for a more detailed treatment of this classical material. 

Let $V$ be an irreducible affine algebraic variety defined over $K$. Let $\displaystyle \U[V]=\U[x]/\I(V/\U)$ denote the coordinate ring of $V$ over $\U$. For each $a\in V$, let $$\M_{V,a}=\{f\in \U[V]\colon f(a)=0\}.$$ Let $r>0$, the \emph{$r$-th jet space of $V$ at $a$}, denoted by $j_r V_a$, is defined as the dual space of the finite dimensional $\U$-vector space $\M_{V,a}/\M_{V,a}^{r+1}$.

The following gives explicit equations for the $r$-th jet space and allows us to consider it as an affine algebraic variety.

\begin{fact}
Let $V\subseteq \U^n$ be an irreducible affine algebraic variety defined over $K$. Fix $r>0$. Let $\DD$ be the set of operators of the form $$\frac{\partial^s}{\partial x_{i_1}^{s_1}\cdots\partial x_{i_k}^{s_k}}$$ where $0<s\leq r$, $1\leq i_1<\cdots<i_k\leq n$, $s_1+\cdots+s_k=s$, and $0<s_i$. Let $a\in V$ and $d=|\DD|$. Then $j_r V_a$ can be identified with the $\U$-vector subspace $$\{(u_D)_{D\in \DD}\in \U^d \colon \sum_{D\in \DD}Df(a)u_D=0, \text{ for all } f\in \I(V/K)\}.$$ 
\end{fact}

Let $X\subseteq V$ be an irreducible algebraic subvariety and $a\in X$. The containment of $X$ in $V$ yields a canonical linear embedding of $j_r X_a$ into $j_r V_a$ for all $r$. We therefore identify $j_r X_a$ with its image in $j_r V_a$.

\begin{fact}\label{thethe}
Let $X$ and $Y$ be irreducible algebraic subvarieties of $V$ and $a\in X\cap Y$. Then, $X=Y$ if and only if $j_r X_a=j_r Y_a$, as subspaces of $j_m V_a$, for all $r$.
\end{fact}

We now introduce the $(\D,\s)$ analogue of the notions of differential and difference modules from \cite{PZ}. 

\begin{definition}
By a \emph{$(\D,\s)$-module over $(\U,\D,\s)$} we mean a triple $(E,\Omega,\Sigma)$ such that $E$ is a $\U$-vector space, $\Omega=\{\dd_1,\dots,\dd_m\}$ is a family of additive endomorphisms of $E$ and $\Sigma$ is an additive automorphism of $E$ such that $$\dd_i(\al e)=\d_i (\al) e+\al \dd_i(e)$$ and $$\Sigma(\al e)=\s(\al)\Sigma(e)$$ for all $\al\in \U$ and $e\in E$, and the operators in $\Omega\cup\{\Sigma\}$ commute.  If we omit $\s$ and $\Sigma$ we obtain Pillay and Ziegler's definition of a \emph{$\D$-module over $(\U,\D)$}. Similarly, if we omit $\D$ and $\Omega$ we obtain the definition of a \emph{$\s$-module over $(\U,\s)$}. 
\end{definition}

The following is for us the key property of $(\D,\s)$-modules.

\begin{lemma}
Let $(E,\Omega,\Sigma)$ be a finite dimensional $(\D,\s)$-module over $\U$. Let $$E^\#=\{e\in E \colon \Sigma(e)=e \text{ and } \dd(e)=0 \text{ for all } \dd\in \Omega\}.$$ Then $E^\#$ is a $\C_\U$-vector space (recall that $\C_\U=\U^\s\cap\U^\D$) and there is a $\C_\U$-basis of $E^\#$ which is also a $\U$-basis of $E$.
\end{lemma}
\begin{proof}
Clearly $E^\#$ is a $\C_\U$-vector space. Let $\{e_1,\dots,e_d\}$ be a $\U$-basis of $E$. With respect to this basis, let $A_i$ be the matrix of $\dd_i$, $i=1,\dots, m$, and $B$ the matrix of $\Sigma$. By this we mean that $A_i$ is the matrix whose $j$-th column consists of the coefficients of the linear combination of $\dd_i(e_j)$ in terms of the basis, and similarly for $B$. Under the linear transformation that takes the basis $\{e_1,\dots,e_d\}$ to the standard basis of $\U^d$, the $(\D,\s)$-module $(E,\Omega,\Sigma)$ is transformed into the $(\D,\s)$-module $$(\U^d,\{\d_1+A_1,\dots,\d_m+A_m\},B\s).$$ It suffices to prove the result for this $(\D,\s)$-module. 
As $\Sigma$ is an additive automorphism of $E$, the matrix $B$ is invertible. Also, the commutativity of $\Omega\cup\{\Sigma\}$ yields:
\begin{equation}\label{uset}
\d_j A_i -\d_i A_j=[A_i,A_j], \quad i=1,\dots,m
\end{equation}
and 
\begin{equation}\label{useg}
B\s(A_i)=\d_i(B)+A_i B, \quad i=1,\dots,m.
\end{equation}
Since $\d_i(B^{-1})=-B^{-1}\d_i(B)B^{-1}$, the previous equation yields
\begin{equation}\label{usec}
B^{-1}A_i=\d_i(B^{-1})+\s(A_i)B^{-1}, \quad i=1,\dots,m.
\end{equation}
Now, note that $(\U^d)^\#=\{u\in \U^d: B \s u=u \text{ and } \d_i u+A_i u=0, i=1,\dots,m\}$, and thus it suffices to find a nonsingular $d\times d$ matrix $M$ over $\U$ such that $B\s M=M$ and $\d_i M+A_i M=0$. Indeed, the columns of this matrix $M$ will form a basis of $\U^d$ whose elements are all in $(\U^d)^\#$. Let $X$ be an $d\times d$ matrix of variables. Extend $\D$ to derivations on $\U(X)$ by letting $\d_i X=-A_i X$ for $i=1,\dots,m$, and $\s$ to an automorphism on $\U(X)$ by letting $\s X=B^{-1}X$. By \ref{uset} the derivations $\D$ on $\U(X)$ commute and by \ref{usec} the automorphism $\s$ on $\U(X)$ commutes with $\D$. Hence, $\U(X)$ is a differential-difference field extension of $\U$ such that $B\s X=X$ and $\d_i X+A_i X=0$. As $(\U,\D,\s)$ is existentially closed we can find a nonsingular matrix $M$ over $\U$ satisfying the desired properties.
\end{proof}

Notice that if $\{e_1,\dots,e_d\}\subset E^\#$ is a $\U$-basis of $E$, which exists by the previous lemma, then, under the linear transformation that takes this basis to the standard basis of $\U^d$, the $(\D,\s)$-module $(E,\Omega,\Sigma)$ is transformed into the $(\D,\s)$-module $(\U^d, \D,\s)$.

\begin{remark}\label{coli}
Let $(E,\Omega)$ be a $\D$-module over $(\U,\D)$ and $E^*$ be the dual space of $E$. If we define the \emph{dual operators} $\Omega^*=\{\dd_1^*,\dots,\dd_m^*\}$ on $E^*$ by $$\dd_i^*(\lambda)(e)=\d_i(\lambda(e))-\lambda(\dd_i(e))$$ for all $\lambda \in E^*$ and $e\in E$, then $(E^*,\Omega^*)$ becomes a $\D$-module over $\U$. Indeed, by Remark~3.3 of \cite{PZ}, $(E^*,\dd_i^*)$ is a $\{\d_i\}$-module. Hence, all we need to verify is that the $\dd_i^*$'s commute: 
\begin{align*}
\dd^*_j(\dd^*_i(\lambda))(e)
&= \d_j(\d_i(\lambda(e)))-\d_j(\lambda(\dd_i(e)))-\d_i(\lambda(\dd_j(e)))+\lambda(\dd_i(\dd_j(e))) \\
&= \d_i(\d_j(\lambda(e)))-\d_i(\lambda(\dd_j(e)))-\d_j(\lambda(\dd_i(e)))+\lambda(\dd_j(\dd_i(e))) \\
&= \dd^*_i(\dd^*_j(\lambda))(e).
\end{align*}
\end{remark}

We now describe a natural $(\D,\s)$-module structure on the jet spaces of an algebraic D-variety equipped with a finite-to-finite correspondence with its $\s$-transform. We first need to recall the notion of an algebraic D-variety (w.r.t. $\D$).

\begin{definition}\label{Dvar}
An (affine) \emph{algebraic D-variety} defined over $K$ is a pair $(V,s)$ where $V\subseteq \U^n$ is an irreducible affine algebraic variety over $K$ and $s=(s_1,\dots,s_m)$ is an $m$-tuple of polynomial maps over $K$ such that each $s_i=(s_i^{(1)},\dots,s_i^{(n)}):V\to \U^n$ satisfies
$$\sum_{k=1}^n\frac{\partial f}{\partial x_k}(x)s_i^{(k)}(x)+f^{\d_i}(x)\in \I(V/K)$$
for every $f\in \I(V/K)$, where $f^{\d_i}$ is the polynomial obtained by applying $\d_i$ to the coefficients of $f$. We also require the following \emph{integrability condition}
\begin{equation}\label{rel1}
\sum_{k=1}^n\frac{\partial s_i^{(\ell)}}{\partial x_k}(x)s_j^{(k)}(x)+s_i^{(\ell)}(x)\equiv  \sum_{k=1}^n\frac{\partial s_j^{(\ell)}}{\partial x_k}(x)s_i^{(k)}(x)+s_j^{(\ell)}(x)\quad \text{mod}\,\I(V/K)
\end{equation}
for all $1\leq i<j\leq m$ and $\ell=1,\dots,n$.
\end{definition}

In the ordinary case $\D=\{\d\}$ this notion coincides with the algebraic D-varieties studied by Pillay in \cite{PiT} and \cite{Pi7}. Moreover, the way we have presented them here is in the spirit of \S3 of \cite{Le3}, where the more general notion of \emph{relative D-variety} was introduced (our algebraic D-varieties are in fact relative D-varieties w.r.t. $\D/\emptyset$).

The set of sharp points of an affine D-variety $(V,s=(s_1,\dots,s_m))$ is defined as 
$$(V,s)^\#=\{a\in V\colon s_i(a)=\d_i a, \, i=1,\dots,m\}.$$
In \cite{Le3} it is shown that the integrability condition~\ref{rel1} is a necessary and sufficient condition for $(V,s)^\#$ to be Zariski dense in $V$. Also, it follows from the equations of the sharp points that if $a\in (V,s)^\#$ then the $\D$-field generated by $a$ over $K$ has finite transcendence degree over $K$ (in fact it is equal to $K(a)$). Conversely, we have

\begin{fact}[\cite{Le3},\S3]\label{interde}
Let $a$ be a tuple such that the $\D$-field generated by $a$ over $K$ has finite transcendence degree over $K$. Then, up to $\D$-interdefinability over $K$,  $a$ is a $\D$-generic point of $(V,s)^\#$ over $K$ for some algebraic D-variety $(V,s)$.
\end{fact}

Here $\D$-interdefinability means interdefinable in the language of differential rings.

Let $(V,s=(s_1,\dots,s_m))$ be an algebraic D-variety defined over $K$. Then we can extend the derivations $\D$ to the coordinate ring $\U[V]$ by defining $\d_i(x)=s_i(x)$ where $x=(x_1,\dots,x_n)$ are the coordinate functions of $\U[V]$. The integrability condition (\ref{rel1}) shows that these extensions commute with each other. Hence, $(\U[V],\D)$ becomes a $\D$-ring (having a $\D$-ring structure on the coordinate ring is the approach of Buium to D-varieties \cite{Buium}).

Let $a\in (V,s)^\#$. Then, $\M_{V,a}^r$ is a $\d_i$-ideal of the $\d_i$-ring $\U[V]$ for all $r>0$. This is shown explicitly in Lemma 3.7 of \cite{PZ}. Hence, $(\M_{V,a}/\M_{V,a}^r,\D)$ becomes a $\D$-module over $(\U,\D)$. By Remark \ref{coli}, $(j_r V_a,\D^*)$ is a $\D$-module over $(\U,\D)$. 

Suppose now $a\in (V,s)^\#$ is a $\D$-generic point over $K$ and $\s a\in K(a)^{alg}$. Let $W$ be the Zariski locus of $(a,\s a)$ over $K$. Then $W\subseteq V\times V^\s$ projects dominantly and generically finite-to-one onto both $V$ and $V^\s$. Moreover, for each $r>0$, $j_r W_{(a,\s a)}\subseteq j_r V_a\times j_rV^\s_{\s a}$ is the graph of an isomorphism $f:j_r V_a\to j_r V^\s_{\s a}$ and the map $\s^*=f^{-1}\comp \s$ equips $j_r V_a$ with the structure of a $\s$-module over $(\U,\s)$ (see Lemma 4.3 of \cite{PZ} for details). Furthermore, Lemma 4.4 of \cite{Bu3} shows that $(j_r V_a,\d_i^*,\s^*)$ is a $(\d_i,\s)$-module over $(\U,\d_i,\s)$ for all $i=1,\dots,m$. Thus, since we have already seen that the dual operators $\D^*$ commute, $(j_r V_a,\D^*,\s^*)$ is a $(\D,\s)$-module over $(\U,\D,\s)$.

\begin{remark}
Let $V$ be an algebraic D-variety defined over $K$ and suppose that $a\in (V,s)^\#$ is a $\D$-generic point over $K$ such that $\s a\in K(a)^{alg}$. Suppose $L>K$ is an algebraically closed differential-difference field and $W$ is the Zariski locus of $a$ over $L$. Then $(W,s|_W)$ is an algebraic D-variety and, under the identification of $j_r W_a$ as a subspace of $j_r V_a$, we have that $j_r W_a$ is a $(\D,\s)$-submodule of $(j_r V_a,\D^*,\s^*)$. Indeed, by Lemma 4.7 of \cite{Bu3}, $j_r W_a$ is a $(\d_i,\s)$-submodule of $(j_r V_a,\d_i^*,\s^*)$ for all $i=1,\dots,r$.
\end{remark}

A type $p=tp(a/K)$ is said to be \emph{finite dimensional} if the transcendence degree of $K\l a\r$ over $K$ is finite. We now show that, up to interdefinability over $K$, if $tp(a/K)$ is finite dimensional, then $\s a\in K(a)^{alg}$ and $a$ is a $\D$-generic point of the sharp points of an algebraic D-variety over $K$.

\begin{lemma}\label{trey}
Suppose $p=tp(a/K)$ is finite dimensional. Then there is an algebraic D-variety $(V,s)$ over $K$ and a $\D$-generic point $c\in (V,s)^\#$ over $K$ such that $K\l a\r=K\l c\r$ and $\s c\in K(c)^{alg}$.
\end{lemma}
\begin{proof}
Since $p$ is finite dimensional then there is $s< \w$ such that $K\l a\r\subseteq K(\T_s a)^{alg}$ where
$$\T_s a=\{\d_m^{e_m}\cdots \d_1^{e_1}\s^k a \colon  e_i,k < \w \text{ and } e_1+\cdots+e_m+k\leq s\}.$$
In particular, $\s^{s+1} a\in K(\T_s a)^{alg}$. Hence, if we let $b=(a,\s a,\dots,\s^s a)$, then $K\l a\r=K\l b\r$ and $\s b$ is algebraic over the $\D$-field generated by $b$ over $K$. Since the latter has finite transcendence degree over $K$, Fact~\ref{interde} implies that there is a tuple $c$, $\D$-interdefinable with $b$ over $K$, such that $c$ is a $\D$-generic point of $(V,s)^\#$ over $K$ for some algebraic D-variety $(V,s)$. Hence, $K\l a\r=K\l c\r$ and $\s c\in K(c)^{alg}$.
\end{proof}

We now give a standard description (up to interalgebraicity over $K$) of the canonical base for finite dimensional types.

\begin{lemma}\label{minf}
Suppose $a$ is a tuple such that $\s a, \d a \in K(a)^{alg}$ for all $\d\in \D$. Let $L>K$ be an algebraically closed differential-difference field, $V$ the Zariski locus of $a$ over $L$, and $F$ the minimal field of definition of $V$. Then, $$Cb(a/L)\subseteq\operatorname{acl}(F,K)\quad \text{ and } \quad F\subseteq Cb(a/L).$$ In particular, $\operatorname{acl}(Cb(a/L), K)=\operatorname{acl}(F,K)$. 
\end{lemma}
\begin{proof}
As $F$ is the minimal field of definition of $V=loc(a/L)$, $a$ is algebraically disjoint from $L$ over $F$. Then, as $\s a, \d a\in K(a)^{alg}$ for all $\d\in \D$, $K\l a\r\subseteq K(a)^{alg}$ is algebraically disjoint from $L$ over $\operatorname{acl}(F, K)$. Hence, $a\ind_{\operatorname{acl}(F, K)} L$. By Proposition~\ref{indal}, the restriction of $tp(a/L)$ to $\operatorname{acl}(F,K)$ is an amalgamation base, and so, by definition, $Cb(a/L)\subseteq \operatorname{acl}(F,K)$.
On the other hand, since $a\ind_{Cb(a/L)} L$, $a$ is algebraically disjoint from $L$ over $Cb(a/L)$. Hence, the Zariski locus of $a$ over $Cb(a/L)$ must also be $V$, and so $V$ is defined over $Cb(a/L)$. Thus, $F\subseteq Cb(a/L)$.
\end{proof}

We are now in position to prove the canonical base property for finite dimensional types. Let us first recall that a type $tp(a/k)$, over a differential-difference $k$, is almost $\C_\U$-internal if there is $b\ind_k a$ and a tuple $c$ from $\C_\U$ such that $a\in \operatorname{acl}(k,b,c)$. 

\begin{theorem}[Canonical base property]\label{cbp}
Suppose $tp(a/K)$ is finite dimensional and $L>K$ is an algebraically closed differential-difference field. Then 
$$tp(Cb(a/L)/K\l a\r)$$ 
is almost $\C_\U$-internal.
\end{theorem}
\begin{proof}
We may replace $a$ by anything interdefinable with it over $K$. Hence, by Lemma~\ref{trey}, we may assume that $\s a\in K(a)^{alg}$ and that there is an algebraic D-variety $(V,s)$ defined over $K$ such that $a\in (V,s)^\#$ is a $\D$-generic point over $K$. Let $W$ be the locus of $a$ over $L$. Then $(W,s|_W)$ is an algebraic $D$-variety with $a$ as a $\D$-generic point of $(W,s|_W)^\#$ over $L$. By Lemma~\ref{minf}, if $F$ is the minimal field of definition of $W$ then $\operatorname{acl}(F,K)=\operatorname{acl}(Cb(a/L),K)$. Thus, it suffices to prove that $tp(F/K\l a\r)$ is almost $\C_\U$-internal.

Consider the $(\D,\s)$-module $(j_r V_a, \D^*,\s^*)$ and recall that $j_r W_a$ is a $(\D,\s)$-submodule. For each $r$, let $b_r$ a $\C_\U$-basis of $j_r V_a^\#$ which is also a $\U$-basis of $j_r V_a$. Let 
$$B=\bigcup_{r=1}^{\infty}b_r,$$ we may choose the $b_r$'s such that $F\ind_{K\l a\r_{\D,\s}} B$. The basis $b_r$ yields a $(\D,\s)$-module isomorphism between $(j_r V_a,\D^*,\s^*)$ and $(\U^{d_r},\D,\s)$ which therefore takes $j_r W_a$ into a $(\D,\s)$-submodule $S_r$ of $(\U^{d_r},\D,\s)$. We can find a $\C_\U$-basis $e_r$ of $S_r^\#\subseteq \C_\U^{d_r}$ which is also a $\U$-basis of $S_r$, so $S_r$ is defined over $e_r\subset\C_\U^{d_r}$. Let $\displaystyle E=\bigcup_{r=1}^{\infty}e_r$. 

It suffices to show that $F\subseteq \operatorname{dcl}(a,K,B,E)$. To see this, let $\phi$ be an automorphism of $(\U,\D,\s)$ fixing $a,K,B,E$ pointwise. Since $j_r V_a$ is defined over $K\l a \r$, then $\phi(j_r V_a)=j_rV_a$. Also, as each $S_r$ is defined over $E$ and the isomorphism between $S_r$ and $j_r W_a$ is defined over $B$, $\phi(j_rW_a)=j_r W_a$ for all $r>0$. Since $V$ is defined over $K$ and $\phi$ fixes $a$ pointwise, Fact~\ref{thethe} implies that $\phi(W)=W$. But $F$ is the minimal field of definition of $W$, thus $\phi$ fixes $F$ pointwise, as desired.
\end{proof}

\begin{remark}\label{onjet}
Even though we did not mentioned it explicitly in the proof of Theorem~\ref{cbp}, the key construction in it is that of
\begin{align*}
(j_r W_a,\D^*,\s^*)^\#
&= \{\lambda \in j_r W_a\colon \s^*(\lambda)=\lambda \text{ and }\d_i^*(\lambda)=0, i=1,\dots,m\} \\
&= \{\lambda\in j_r W_a\colon \s(\lambda)=f(\lambda) \text{ and } \d_i\comp \lambda=\lambda\comp \d_i, i=1\dots,m\},
\end{align*}
which is a finite dimensional $\C_\U$-vector space. One could call this vector space the $r$-th $(\D,\s)$-jet space at $a$ of the $(\D,\s)$-locus of $a$ over $L$. However, $(\D,\s)$-jet spaces for arbitrary $(\D,\s)$-varieties have already been defined, they are a special case of the generalized Hasse-Schmidt jet spaces of Moosa and Scanlon from \cite{MS2}, and so there is the question of how our construction compares to theirs. In the appendix we show that in this context, namely $a$ is a $\D$-generic point of the sharp points of an algebraic D-variety over $L$ and $\s a \in L(a)^{alg}$, the two constructions agree. (This has been checked in the differential case, see the end of \S4.3 of \cite{MS2}.) Thus, it is consistent and appropriate to use this terminology to refer to $(j_r W_a,\D^*,\s^*)^\#$. It is worth mentioning that this fact is independent from the results of this paper, and that we spell it out (in the appendix) for the benefit of the reader familiar with the work of Moosa and Scanlon.
\end{remark}

As a consequence of the canonical base property, we obtain the Zilber dichotomy for finite dimensional minimal types. We omit the standard proof (for detailed proofs in the settings of $DCF_0$ and $\mathcal{D}$-$CF_0$ we refer the reader to Corollary 3.10 of \cite{PZ} and Corollary 6.19 of \cite{MS}, respectively).

\begin{corollary}[Zilber dichotomy]\label{zildi}
Let $p$ be a finite dimensional complete type over $K$ with $SU(p)=1$. Then $p$ is either one-based or almost $\C_\U$-internal.
\end{corollary}

\begin{remark}
The assumption in Corollary~\ref{zildi} that $p$ is finite dimensional should not be necessary, though a different proof is needed. To prove the general case one could work out the theory of arc spaces of Moosa, Pillay, Scanlon \cite{MPS} in the $(\D,\s)$ setting, and apply their weak dichotomy for regular types. This is done in the $(\d,\s)$ setting (i.e. $DCF_{0,1}A$) by Bustamante in \cite{Bu4}, and we expect that the arguments there extend to our setting. 
\end{remark}

\section*{Appendix: On $(\D,\s)$-jet spaces}

As in Section \ref{Zil}, we fix a sufficiently saturated $(\U,\D,\s)\models DCF_{0,m}A$, and an algebraically closed differential-difference subfield $K$ of $\U$. In this appendix we show that if $(V,s)$ is an algebraic D-variety over $K$ (as in Definition~\ref{Dvar}) and $a$ is a $\D$-generic point of $(V,s)^\#$ over $K$ such that $\s a\in K(a)^{alg}$, then $(j_r V_a,\D^*,\s^*)^\#$ is the $r$-th $\HS$-jet space at $a$ of the $\HS$-locus of $a$ over $K$, where $\HS$ is the iterative Hasse-Schmidt system for differential-difference rings described in Example A.2 of \cite{MS2}. 

As we mentioned in Remark~\ref{onjet}, the results in this appendix are independent from the results of the rest of the paper, and we present them to justify why the construction of $(j_r V_a,\D^*,\s^*)^\#$ is a rigorous one and indeed yields the right notion of the $r$-th $(\D,\s)$-jet space at $a$ of the $(\D,\s)$-locus of $a$ over $K$. Throughout this appendix we assume that the reader is familiar with the theory of generalized Hasse-Schmidt jet spaces developed by Moosa and Scanlon in \cite{MS1} and \cite{MS2}.

Let us start by recalling the iterative Hasse-Schmidt system $\HS=\left(\DD_n\right)_{n<\w}$ for (partial) differential-difference rings described in \cite{MS2}. Let $\eta=(\eta_1,\dots,\eta_m)$ be an $m$-tuple of indeterminates. For each $\QQ$-algebra $R$, let 
$$\DD_n(R)=\prod_{i=1}^n R[\eta]/(\eta)^{i}\times R[\eta]/(\eta)^{n+1} \times \prod_{i=1}^n R[\eta]/(\eta)^{n+1-i},$$
$s_n:R\to \DD_n(R)$ be the diagonal compose with the natural inclusion on each factor, $\psi_n:\DD_n(R)\to R^{\ell_n}$ be the identification via a fixed monomial basis on each factor, and $\pi_{m,n}:\DD_m(R)\to\DD_n(R)$ be the projection from the $(m-n)$ to the $(m+n+1)$ coordinates compose with the natural quotient on each factor. The iteration map $\Ga_{(m,n)}: \DD_{m+n}(R)\to \DD_m(\DD_n(R))$ is given by 
$$f_i(\eta)_{-n-m\leq i\leq n+m}\mapsto ((f_{\alpha+\beta}(\chi +\epsilon))_{-n\leq \alpha\leq n})_{-m\leq \beta\leq m},$$
where $\chi$ and $\epsilon$ are $m$-tuples of indeterminates.

Given a differential-difference field $(F,\D,\s)$, one can make $F$ into an iterative $\HS$-field by setting $E_n:F\to \DD_n(F)$ to be 
\begin{align*}
E_n(x)
&= \left(\sum_{\alpha\in \NN^m,|\alpha|\leq n-i} \frac{1}{\alpha !}\s^{-i}\d^{\alpha}(x)\eta^{\alpha},\sum_{\alpha\in \NN^m,|\alpha|\leq n-i} \frac{1}{\alpha !}\d^{\alpha}(x)\eta^{\alpha}, \right. \\
&\quad  \quad\left.\sum_{\alpha\in \NN^m,|\alpha|\leq n-i} \frac{1}{\alpha !} \s^i \d^{\alpha}(x)\eta^{\alpha} :\; i=1,\dots,n\right),
\end{align*}
where $\d=(\d_1,\dots,\d_m)$ and multi-index notation is being used.
Conversely, an iterative $\HS$-field is naturally a differential-difference field. Indeed, if $(F,E)$ is a $\HS$-field, where $E=(E_n:F\to \DD_n(F))_{n<\w}$ is the sequence of $\QQ$-algebra homomorphisms, then the maps $\d_1,\dots,\d_m,\s:K\to K$ coming from 
$$E_1(x)=(\s^{-1} (x), x+\sum_{i=1}^m \d_i(x)\eta_i,\s (x))$$
yield a set of commuting derivations and an automorphim commuting with the derivations.

Our main tool will be the following characterization of $\HS$-jet spaces:

\begin{fact}[\cite{MS2}, Theorem 4.17]\label{usef}
Suppose $\underline{Z}$ is an irreducibe $\HS$-subvariety of an algebraic variety $X$ over $K$, and $a\in \underline{Z}(\U)$ is a generic point over $K$. An algebraic jet $\lambda\in j_r X_a(\U)$ is in $j_r^{\HS}\underline{Z}_a(\U)$ if and only if for all $k\geq 0$ there exists $\gamma_k\in j_r (Z_k)_{\nabla_k(a)}(\U)$ extending $\lambda$, such that $E_k\comp \lambda=(\gamma_k\otimes_\U \DD_k(\U) )\comp e_k$.
\end{fact}

Here $\nabla_k: X(\U)\to \tau_k X(\U)$ is the nabla map and 
$$e_k:\M_{Z_0,a}/\M_{Z_0,a}^{r+1}\to \M_{Z_k,\nabla_k(a)}/\M_{Z_k,\nabla_k(a)}^{r+1}\otimes_\U \DD_k(\U)$$
denotes the additive map induced by the canonical morphism $r_k^{Z_0}:\tau_k Z_0\times_\U \DD_k(\U)\to Z_0$ associated to the prolongation $\tau_k Z_0$.

Now suppose that $(V,s)$ is an algebraic D-variety over $K$ and $a$ is a $\D$-generic point of $(V,s)^\#$ over $K$ such that $\s a\in K(a)^{alg}$. Let $\underline{Z}=(Z_k:k\in \NN)$ be the $\HS$-locus of $a$ over $K$; that is, $Z_k$ is the Zariski locus of $\nabla_k(a)$ over $K$. Our goal is to show that
$$j_r^{\HS}\underline{Z}_a(\U)=(j_r V_a,\D^*,\s^*)^\#,$$ 
for all $r>0$.

As $a$ is a $\D$-generic point of $(V,s)^\#$, $a$ is a (Zariski) generic point of $V$. Hence, $Z_0=V$.
Since $\bar s=(\operatorname{Id},s)$ is an isomorphism between $V$ and the Zariski locus of $(a,\d a)=(a,\d_1 a,\dots,\d_m a)$ over $K$, we have that $Z_1$ decomposes as the fibred product
$$Z_1=W^{\s^{-1}}\times_V \bar{s}(V) \times_V W,$$
where $W$ is the Zariski locus of $(a,\s a)$. Moreover, for each $r>0$, $j_r(Z_1)_{\nabla_1(a)}$ decomposes as
$$j_r(Z_1)_{\nabla_1(a)}=j_r W^{\s^{-1}}_{(\s^{-1}a,a)}\times_{j_r V_a} j_r \bar{s}(V)_{(a,\d a)}\times_{j_r V_a} j_r W_{(a,\s a)}.$$
Recall that $j_r W_{(a,\s a)}$ is the graph of the isomorphism $f:j_rV_a\to j_r V^\s _{\s a}$ used to equip $j_r V_a$ with the $\s$-module structure $\s^*=f^{-1}\comp \s$. Also, one can check that $j_r W^{\s^{-1}}_{(\s^{-1}a,a)}$ is the graph of the isomorphism $g:j_r V^{\s^{-1}}_{\s^{-1}a}\to j_r V_a$ given by $g=\s^{-1}\comp f\comp \s$. 

Suppose $\lambda \in j_r V_a$. Then $\gamma_1:=(g^{-1}(\lambda), j_r\bar{s}_a(\lambda),f(\lambda))$ is the unique element of $j_r (Z_1)_{\nabla_1(a)}$ extending $\lambda$. Furthermore, since $\underline{Z}(\U)$ is given by first order differential and difference equations, the criterion of Fact \ref{usef} reduces to only checking the case $k=1$. Thus, $\lambda\in j_r^{\HS}\underline{Z}_a(\U)$ if and only if 
\begin{equation}\label{fina}
E_1\comp \lambda=(\gamma_1\otimes \mathcal{D}_1(\U))\comp e_1.
\end{equation}

Let $e\in \M_{V,a}/\M_{V,a}^{r+1}$. By definition, 
$$E_1\comp\lambda(e)=\left(\s^{-1}\comp\lambda(e),\lambda(e)+\sum_{i=1}^m \d_i\comp\lambda(e) \eta_i,\s\comp\lambda (e)\right).$$
Also, by definition,
$$e_1(e)=\left(e^{\s^{-1}},\, e+\sum_i\left(\sum_j \frac{\dd e}{\dd x_j}y_j+e^{\d_i}\right)\eta_i,\, e^\s\right).$$
Hence, $(\gamma_1\otimes\mathcal{D}_1(\U))\comp e_1(e)$ is equal to
$$\left(g^{-1}(\lambda)(e^{\s^{-1}}),\, \lambda\left(e+\sum_i\left(\sum_j \frac{\dd e}{\dd x_j}s_j+e^{\d_i}\right)\eta_i\right),\, f(\lambda)(e^\s)\right),$$
which, by definition of the extension of $\D$ to $\U[V]$, is equal to 
$$\left(g^{-1}(\lambda)(e^{\s^{-1}}),\, \lambda(e)+\sum_i \lambda\comp\d_i(e)\eta_i,\, f(\lambda)(e^\s)\right).$$
Using the fact that $\s\comp\lambda(e)=\s(\lambda)(e^\s)$ and that $g=\s^{-1}\comp f\comp\s$, we get that $f(\lambda)(e^\s)=\s\comp \lambda (e)$ is equivalent to $g^{-1}(\lambda)(e^{\s^{-1}})=\s^{-1}\comp \lambda (e)$. Thus, equality (\ref{fina}) is equivalent to 
$$\s(\lambda)=f(\lambda)\; \text{ and } \; \d_i\comp \lambda=\lambda\comp \d_i,\, i=1,\dots,m.$$
These are precisely the equations defining $(j_r V_a,\D^*,\s^*)^\#$, see Remark \ref{onjet}.

\

\bibliographystyle{plain}

\begin{thebibliography}{10}

\bibitem{Buium}
A. Buium.
\newblock Differential algebraic groups of finite dimension.
\newblock Lecture notes in mathematics, Vol. 1506, Springer-Verlag, 1992.

\bibitem{Bu}
R. Bustamante.
\newblock Differentially closed fields of characteristic zero with a generic automorphism.
\newblock Revista de Matem\'atica: Teor\'ia y Aplicaciones 14(1), pp.81-100, 2007.

\bibitem{Bu2}
R. Bustamante. 
\newblock Rank and dimension in difference-differential fields. 
\newblock Notre Dame Journal of Formal Logic, Vol. 52 (4), pp.403-414, 2011.

\bibitem{Bu3}
R. Bustamante. 
\newblock Algebraic jet spaces and Zilber's dichotomy in DCFA. 
\newblock Revista de Matem\'atica: Teor\'ia y Aplicaciones 17 (1), pp.1-12, 2010.

\bibitem{Bu4}
R. Bustamante.
\newblock Th\'eorie des mod\`eles des corps diff\'erentiellement clos avec un automorphisme g\'enerique.
\newblock PhD Thesis. Universit\'e Paris 7, 2005.

\bibitem{Ca}
E. Casanovas.
\newblock Simple theories and hyperimaginaries.
\newblock Lecture Notes in Logic. Cambridge University Press, 2011.

\bibitem{CH}
Z. Chatzidakis and E. Hrushovski. 
\newblock Model theory of difference fields. 
\newblock Trans. Amer. Math. Soc.351 (8), pp.2997-3071, 1999.

\bibitem{CP}
Z. Chatzidakis and A. Pillay.
\newblock Generic structures and simple theories.
\newblock Annals of Pure and Applied Logic 95, pp.71-92, 1998.

\bibitem{Co}
R. Cohn.
\newblock A difference-differential basis theorem.
\newblock Canadian J. Math. 22, No. 6, pp.1224-1237, 1970.

\bibitem{GKO}
O. Golubitsky, M. Kondratieva and A. Ovchinnikov 
\newblock On the generalised Ritt problem as a computational problem. 
\newblock Journal of Math. Sci. 163 (5), pp.515-522, 2009.

\bibitem{GR}
N. Guzy and C. Riviere.
\newblock On existentially closed partial differential fields with an automorphism.
\newblock Preprint 2007 (http://w3.umh.ac.be/math/logic/sources/pdcfa-cras.pdf)

\bibitem{HaSi}
C. Hardouin and M. Singer.
\newblock Differential Galois theory of linear difference equations.
\newblock Mathematische Annalen 342(2), pp.333-377, 2008.

\bibitem{HKM}
M. Harrison-Trainor, J. Klys and R. Moosa. 
\newblock Nonstandard methods for bounds in differential polynomial rings. 
\newblock Journal of Algebra, Vol. 360, pp.71-86, 2012.

\bibitem{HKP}
B. Hart, B. Kim and A. Pillay.
\newblock Coordinatisation and canonical basis in simple theories.
\newblock Journal of Symbolic Logic. Vol. 65, No. 1, pp.293-309, 2000.

\bibitem{Hi}
M. Hils. 
\newblock Generic automorphisms and green fields. 
\newblock J. Lond. Math. Soc. (2) 85, No. 1, pp.223-244, 2012. 

\bibitem{Hru1}
E. Hrushovski.
\newblock The Mordell-Lang conjecture for function fields.
\newblock J. Amer. Math. Soc., 9, No. 3, pp.667-690, 1996.

\bibitem{Hru2}
E. Hrushovski.
\newblock The Manin-Mumford conjecture and the model theory of difference fields.
\newblock Annals of Pure and Applied Logic, 112(1), pp.43-115, 2001.

\bibitem{Ko}
E. Kolchin.
\newblock Differential algebra and algebraic groups.
\newblock Academic Press. New York, New York, 1973.

\bibitem{Ko2}
E. Kolchin.
\newblock Differential algebraic groups.
\newblock Academic Press, Inc. 1985.

\bibitem{Kow}
P. Kowalski. 
\newblock Geometric axioms for existentially closed Hasse fields. 
\newblock Annals of Pure and Applied Logic 135, pp.286-302, 2005.

\bibitem{Le}
O. Le\'on S\'anchez. 
\newblock Geometric axioms for differentially closed fields with several commuting derivations. 
\newblock Journal of Algebra, Vol. 362, pp.107-116, 2012.

\bibitem{Le2}
O. Le\'on S\'anchez. 
\newblock Corrigendum to ``Geometric axioms for differentially closed fields with several commuting derivations''. 
\newblock Journal of Algebra, Vol. 382, pp.332-334, 2013.

\bibitem{Le3}
O. Le\'on S\'anchez. 
\newblock Relative D-groups and differential Galois theory in several derivations. 
\newblock To appear in Trans. Amer. Math. Soc.


\bibitem{Mc}
T. McGrail. 
\newblock The model theory of differential fields with finitely many commuting derivations.
\newblock Journal of Symbolic Logic, Vol. 65 (2), pp.885-913, 2000.

\bibitem{MS}
R. Moosa and T. Scanlon. 
\newblock Model theory of fields with free operators in characteristic zero.
\newblock Preprint 2012 (http://arxiv.org/pdf/1212.5838v2.pdf)

\bibitem{MS1}
R. Moosa and T. Scanlon.
\newblock Jet and prolongation spaces.
\newblock Journal of the Inst. of Math. Jussieu 9, pp.391-430, 2010.

\bibitem{MS2}
R. Moosa and T. Scanlon.
\newblock Generalized Hasse-Schmidt varieties and their jet spaces.
\newblock Proc. London Math. Soc. 3, pp.197-234, 2011.

\bibitem{MPS}
R. Moosa, A. Pillay and T. Scanlon.
\newblock Differential arcs and regular types in differential fields.
\newblock J. reine angew. Math. 620, pp.35-54, 2008.

\bibitem{PiPi}
D. Pierce and A. Pillay. 
\newblock A note on the axioms for differentially closed fields of characteristic zero. 
\newblock Journal of Algebra, Vol. 204, pp.108-115, 1998.

\bibitem{PZ}
A. Pillay and M. Ziegler.
\newblock Jet spaces of varieties over differential and difference fields.
\newblock Selecta Mathematica. New series 9, pp.579-599, 2003.

\bibitem{PiT}
A. Pillay.
\newblock Two remarks on differential fields.
\newblock Model Theory and Applications. Quaderni di Matematica. Volume 11, 2003. Edited by Dipartimento di Matematica, Seconda Universita di Napoli.

\bibitem{Pi7}
A. Pillay.
\newblock Remarks on algebraic D-varieties and the model theory of differential fields.
\newblock Lecture Notes in Logic 26, Logic in Tehran, pp.256-269, 2006.

\bibitem{Pi8}
A. Pillay.
\newblock Model-theoretic consequences of a theorem of Campana and Fujiki. 
\newblock Fund. Math. 174, pp.187-192, 2002.

\bibitem{Pi9}
A. Pillay.
\newblock Differential Galois theory I. 
\newblock Illinois Journal of Mathematics. Vol. 42, N. 4, 1998.

\bibitem{Ro}
A. Rosenfeld.
\newblock Specializations in differential algebra.
\newblock Trans. Amer. Math. Soc. 90, pp. 394-407, 1959.

\bibitem{Tr}
M. Tressl.
\newblock The uniform companion for large differential fields of characteristic 0.
\newblock Trans. Amer. Math. Soc. 357 (10), pp. 3933-3951, 2005.

\bibitem{VS}
L. van den Dries and K. Schmidt. 
\newblock Bounds in the theory of polynomial rings over fields. A nonstandard approach.
\newblock Invent. Math. 76 (1), pp.77-91, 1984.

\end{thebibliography}

\end{document}